\documentclass[12pt, reqno]{amsart}
\usepackage{amsmath, amsthm, amscd, amsfonts, amssymb, graphicx, color}
\usepackage[bookmarksnumbered, colorlinks, plainpages]{hyperref}
\hypersetup{colorlinks=true,linkcolor=red, anchorcolor=green, citecolor=cyan, urlcolor=red, filecolor=magenta, pdftoolbar=true}

\DeclareMathSymbol{\subsetneqq}{\mathbin}{AMSb}{36}

\newcommand{\R}{\mathbb{R}}

\newcommand{\N}{\mathbb{N}}

\newcommand{\C}{\mathbb{C}}

\newcommand{\beq}{\begin{eqnarray}}
\newcommand{\eeq}{\end{eqnarray}}
\newcommand{\bq}{\begin{equation}}
\newcommand{\eq}{\end{equation}}
\newcommand{\beqn}{\begin{eqnarray*}}
\newcommand{\eeqn}{\end{eqnarray*}}
\newcommand{\bex}{\begin{exo}}
\newcommand{\eex}{\end{exo}}
\newcommand{\ben}{\begin{enumerate}}
\newcommand{\een}{\end{enumerate}}

\newtheorem{th1}{{\bf Theorem}}[section]
\newtheorem{thm}[th1]{{\bf Theorem}}
\newtheorem{lem}[th1]{{\bf Lemma}}
\newtheorem{prop}[th1]{{\bf Proposition}}
\newtheorem{cor}[th1]{{\bf Corollary}}
\newtheorem{rem}[th1]{\bf Remark}
\newtheorem{rems}[th1]{\bf Remarks}
\newtheorem{defi}[th1]{\bf Definition}

\author[T. Saanouni]{Tarek Saanouni}
\address{Department of Mathematics, College of Science and Arts in Uglat Asugour, Qassim University, Buraydah, Kingdom of Saudi Arabia.}
\email{\sl T.saanouni@qu.edu.sa}
\email{\sl Tarek.saanouni@ipeiem.rnu.tn}

\subjclass[2010]{35Q55}
\keywords{Fractional Inhomogeneous Choquard equation, global existence, blow-up.}

\title[Choquard INLS]{Remarks on the fractional inhomogeneous Hartree equation}

\date{\today}
\begin{document}
\begin{abstract}
This paper studies the inhomogeneous fractional Sch\"odinger equation
$$i\dot u-(-\Delta)^s  u=\pm(I_\alpha *|\cdot|^b|u|^p)|x|^b|u|^{p-2}u.$$
In the mass super-critical and energy sub-critical regimes, using a Gagliardo-Nirenberg adapted to the above problem, the standing waves give a sharp threshold of global existence versus finite time blow-up of solutions.
\end{abstract}
\maketitle
\vspace{ 1\baselineskip}
\renewcommand{\theequation}{\thesection.\arabic{equation}}
\section{Introduction}
It is the purpose of this note, to investigate the Cauchy problem for a fractional inhomogeneous Schr\"odinger equation of Choquard type
\begin{equation}
\left\{
\begin{array}{ll}
i\dot u-(-\Delta)^s  u=\epsilon (I_\alpha *|\cdot|^b|u|^p)|x|^b|u|^{p-2}u ;\\
u(0,.)=u_0.
\label{S}
\end{array}
\right.
\end{equation}
The fractional Schr\"odinger equation is a fundamental equation of fractional quantum mechanics \cite{lskn}.
In three space dimensions, if $s=\frac12, b=0$ and $\alpha=p=2$, the above equation arises as an effective description of pseudo-relativistic Boson stars \cite{ll}.\\

Here and hereafter $u$ is a complex valued function of the variable $(t,x)\in\R\times\R^N$, for an integer $N\geq2$. The real number $\epsilon=\pm1$ refers to the defocusing versus focusing regime. The fractional Laplacian operator is
$$(-\Delta)^s\,.\,:=\mathcal F^{-1}(|\xi|^{2s}\mathcal F\,.\,),\quad s\in(0,1).$$
The unbounded inhomogeneous term is $|\cdot|^b$ for a real number $b<0$ and the Riesz-potential is the radial function defined on $\R^N$ as follows 
$$I_\alpha:=\frac{\Gamma(\frac{N-\alpha}2)}{\Gamma(\frac\alpha2)\pi^\frac{N}22^\alpha|\cdot|^{N-\alpha}}:=\frac{\mathcal K}{|\cdot|^{N-\alpha}},\quad  0<\alpha<N.$$
In all this note, one assumes the next restriction on the different parameters of the above problem,
\begin{equation}\label{cnd}
\min\{-b,\alpha,N-\alpha,N+b-s,2s+2b+\alpha,N+\alpha+2b-2s\}>0.
\end{equation}

If $u$ a solution to the above problem, so is the scaled function
$$u_\lambda=\lambda^\frac{2s+2b+\alpha}{2(p-1)}u(\lambda^{2}.,\lambda .),\quad\lambda>0.$$
The critical exponent is the unique real number conserving the homogeneous Sobolev norm 
$$\|u_\lambda(t)\|_{\dot H^{s_c}}=\|u(\lambda^2t)\|_{\dot H^{s_c}},\quad s_c:=\frac N2-\frac{2s+2b+\alpha}{2(p-1)}.$$

In this note, one focus on the mass super-critical $(s_c>0)$ and energy sub-critical $(s_c<1)$ regimes.\\

The inhomogeneous fractional Schr\"odinger problem was considered recently \cite{pz}. Indeed, a sharp dichotomy of global existence versus finite time blow-up of solutions was obtained in the mass super-critical and energy sub-critical regimes. This study follows the ideas of \cite{hr}. It is the aim of this work to extend the previous results to the inhomogeneous fractional Choquard problem \eqref{S}. This note is also a generalization of the previous paper \cite{st2}.\\    
 
It is the contribution of this manuscript, to overcome three difficulties. The first one is the existence of a fractional Schr\"odinger operator, which is partially resolved by considering the radial case. In order to obtain the existence of non-global solutions, one uses a localized variance identity \cite{bhl}. The second one is the presence of an unbounded homogeneous term. The last one is the non-local source term.\\ 

 The rest of this paper is organized as follows. The next section contains some technical tools needed in the sequel. The existence of ground states is proved in section three. Section 4 contains a proof of a sharp Gagliardo-Nirenberg type inequality. The existence of energy sub-critical solutions is established in the fifth section. In section six, a variance type estimate is obtained. A sharp dichotomy of global/non global existence of solutions is given in the last section. \\

We mention that $C$ will denote a constant which may vary from line to line. If $A$ and $B$ are non-negative real numbers, $A\lesssim B$  means that $A\leq CB$. \\
Denote for simplicity the Lebesgue space $L^r:=L^r({\R^N})$ with the usual norm $\|\cdot\|_r:=\|\cdot\|_{L^r}$ and $\|\cdot\|:=\|\cdot\|_2$. Take $H^s:=H^s({\R^N})$ be the usual inhomogeneous Sobolev space endowed with the complete norm 
$$ \|\cdot\|_{H^s} := \Big(\|\cdot\|^2 + \|(-\Delta)^{\frac s2}\cdot\|^2\Big)^\frac12.$$
If $X$ is an abstract space $C_T(X):=C([0,T],X)$ stands for the set of continuous functions valued in $X$ and $X_{rd}$ is the set of radial elements in $X$, moreover for an eventual solution to \eqref{S}, $T^*>0$ denotes it's lifespan. Finally, $x^\pm$ are two real numbers near to $x$ satisfying $x^+>x$ and $x^-<x$.
\section{Background Material}
In this section, one gives the main results and some standard tools needed in the sequel.
\subsection{Preliminary}
The mass-critical and energy-critical exponents are
$$p_*:=1+\frac{\alpha+2s+2b}N,\quad p^*:=\left\{
\begin{array}{ll}
1+\frac{2s+2b+\alpha}{N-2s}\quad\mbox{if}\quad N\geq3;\\
\infty\quad\mbox{if}\quad  N =1,2.
\end{array}
\right.$$
 Here and hereafter define the real numbers
\begin{gather*}
B:=\frac{Np-N-\alpha-2b}s,\quad A:=2p-B;\\
\tilde p:=1+\frac{2b+\alpha}N,\quad \bar p:=1+\frac{2b+\alpha}{N-2s}.
\end{gather*}
For $u\in H^s$, take the quantities
\begin{gather*}
S(u):=\|u\|_{H^s}^2-\frac1p\int_{\R^N}(I_\alpha*|\cdot|^b|u|^p)|x|^b|u|^p\,dx;\\
K(u):=\frac{4s}N\Big(\|(-\Delta)^{\frac s2}u\|^2-\frac B{2p}\int_{\R^N}(I_\alpha*|\cdot|^b|u|^p)|x|^b|u|^p\,dx\Big);\\
H(u):=S(u)-\frac{N}{4s}K(u)=\|u\|^2+\frac{B-2}{2p}\int_{\R^N}(I_\alpha*|\cdot|^b|u|^p)|x|^b|u|^p\,dx.
\end{gather*}
\begin{defi}
A ground state of \eqref{S} is a solution to 
\begin{equation}\label{grnd}
(-\Delta)^s \phi+\phi=(I_\alpha*|\cdot|^b|\phi|^p)|x|^b|\phi|^{p-2}\phi,\quad0\neq\phi\in H^s,
\end{equation}
which minimizes the problem
\begin{equation}\label{min}
m:=\inf_{0\neq u\in H^s}\Big\{S(u) \quad\mbox{s\,. t}\quad K(u)=0\Big\}.
\end{equation}
\end{defi}
If $\phi$ is a ground state to \eqref{S}, the following scale invariant quantities describe the dichotomy of global/non-global existence of solutions \cite{gz}.
$$\mathcal{ME}(u):=\frac{E(u)^{s_c}M(u)^{s-s_c}}{E(\phi)^{s_c}M(\phi)^{s-s_c}}\quad\mbox{and}\quad \mathcal G(u):=\frac{\|(-\Delta)^{\frac s2} u\|^{s_c}\|u\|^{s-s_c}}{\|(-\Delta)^{\frac s2}\phi\|^{s_c}\|\phi\|^{s-s_c}}.$$

Take $\psi\in C_0^\infty(\R^n)$ is a radial function satisfying $\psi''\leq1$ and
$$\psi(x)=\left\{
\begin{array}{ll}
\frac12|x|^2,\quad |x|\leq1 ;\\
0,\quad |x|\geq2.
\end{array}
\right.$$
Then, $\psi_R:=R^2\psi(\frac\cdot R)$ satisfies
$$\psi_R''\leq1,\quad \psi_R'(r)\leq r\quad\mbox{and}\quad \Delta\psi_R\leq N.$$
Denote the localized variance
$$M_\psi[u]:=2\Im\int_{\R^N}u\nabla\psi·\nabla u\,dx = 2\Im\int_{\R^N}u\partial_k\psi\partial_k u\,dx. $$
Finally, define the self-adjoint differential operator
$$\Gamma_\psi :=-i\Big[\nabla((\nabla\psi)\cdot) + \nabla\psi \nabla\cdot\Big].$$
The next sub-section contains the contribution of this manuscript.
\subsection{Main results}
First, one considers the stationary problem associated to \eqref{S} and investigates the existence of ground states.
\begin{thm}\label{t3}
Take $N\geq2$, $s\in(0,1)$, $b,\alpha$ satisfying \eqref{cnd} and $p_* <p<p^*$. Then, there is a ground state solution to \eqref{grnd}-\eqref{min}.
\end{thm}
The next inhomogeneous Gagliardo-Nirenberg type inequality adapted to the problem \eqref{S}, will be useful in this note.
\begin{thm}\label{gag}
Let $N\geq2$, $s\in(0,1)$, $b,\alpha$ satisfying \eqref{cnd} and $\tilde p< p< p^*$. Then, 
\begin{enumerate}
\item[1.]
there exists $C(N,p,b,\alpha,s)>0$, such that for any $u\in H^s$,
\begin{equation}\label{ineq}
\int_{\R^N}(I_\alpha*|\cdot|^b|u|^p)|x|^b|u|^p\,dx\leq C(N,p,b,\alpha,s)\|u\|^A\|(-\Delta)^{\frac s2} u\|^B.
\end{equation}
\item[2.]
there exists $\psi\in H^s$ a minimizing of the problem
$$\frac1{C(N,p,b,\alpha,s)}=\inf\Big\{J(u):=\frac{\|u\|^A\|(-\Delta)^{\frac s2} u\|^B}{\int_{\R^N}(I_\alpha*|\cdot|^b|u|^p)|x|^b|u|^p\,dx},\quad0\neq u\in H^s\Big\}$$
 such that ${C(N,p,b,\alpha,s)}=\int_{\R^N}(I_\alpha*|\cdot|^b|\psi|^{p})|x|^b|\psi|^p\,dx$ and
\begin{equation}\label{euler}
B(-\Delta)^s\psi+A\psi-\frac{2p}{C(N,p,b,\alpha,s)}(I_\alpha*|\cdot|^b|\psi|^p)|x|^b|\psi|^{p-2}\psi=0;
\end{equation}
\item[3.]
if $\phi$ is a ground state to \eqref{S}, then
\begin{equation}\label{part3}
C(N,p,b,\alpha,s)=\frac{2p}{A}(\frac AB)^{\frac{B}2}\|\phi\|^{-2(p-1)}.
\end{equation}
\end{enumerate}
\end{thm}
Second, one considers the evolution problem \eqref{S} and obtains a local well-posed result in the energy space.
\begin{thm}\label{t0}
Let $N\geq2$, $\frac{N}{2N-1}\leq s<1$, $b,\alpha$ satisfying \eqref{cnd} and $N<4s+\alpha+2b$, $u_0\in H^s$ and $\max\{2,\bar p\}< p< p^*$. Then, there exists $T^* = T^*(\|u_0\|_{H^s})$ such that \eqref{S} admits a unique maximal solution
$$ u\in C_{T^*}(H^s),$$ 
which satisfies the conservation laws
\begin{gather*}
Mass:=M(u(t)) :=\int_{\R^N}|u(t,x)|^2dx = M(u_0);\\
Energy:=E(u(t)) :=\|(-\Delta)^{\frac s2}u(t)\|^2+\frac\epsilon p\int_{\R^N}|x|^b(I_\alpha *|\cdot|^b|u(t)|^p)|u(t)|^p\,dx= E(u_0). 
\end{gather*}
\end{thm}
\begin{rems}
\begin{enumerate}
\item[1.] $u\in L^q_{loc}((0,T^*),W^{1,r})$ for any admissible pair $(q,r)$;
\item[2.]
$T^*=\infty$ in the defocusing case or mass-sub-critical case;
\item[3.]
the previous theorem seems to hold for $\tilde p< p \leq p^*$. 
\end{enumerate}
\end{rems}
The next localized variance identity, will be needed to obtain the existence of non-global solutions to \eqref{S}.
\begin{prop}\label{viril}
Let $N\geq2$, $s\in(\frac12, 1)$, $b,\alpha$ satisfying \eqref{cnd} and $\tilde p<p<p^*$. Take $u\in C_T(H^s_{rd})$ be a local solution of \eqref{S}. Then, for any $R>0$ and $\varepsilon>0$ near to zero, holds on $[0,T)$, 
\begin{eqnarray*}
\frac{d}{dt}M_{\psi_R}[u]
&\leq&2sBE-2s(B-2)\|(-\Delta)^{\frac s2}u\|^2+\frac C{R^{2s}}\\
&+&\frac C{R^{(N-1-\varepsilon-2b)(p-1-\frac\alpha N)}}\|(-\Delta)^\frac s2 u\|^{\frac{1+\varepsilon}{s}(p-1-\frac\alpha N)}.
\end{eqnarray*}
\end{prop}
The standing waves give a threshold of global existence versus finite time blow-up of solutions.
\begin{thm}\label{Blow-up}
Let $ N\geq2$, $\frac{N}{2N-1}\leq s<1$, $b,\alpha$ satisfying \eqref{cnd}, $u_0\in H^s$, $0< s_c<s$, $\phi$ be a ground state solution to \eqref{grnd} satisfying
\begin{equation} \label{ss}
\mathcal{ME}(u_0)<1.
\end{equation}
Take a maximal solution ${u}\in C_{T^*}(H^s)$ of \eqref{S}. Thus,
\begin{enumerate}
\item[1.]
if 
\begin{equation} \label{ss1}
\mathcal{G}(u_0)<1,
\end{equation}
 then, ${u}$ is global;
\item[2.]
if $p<1+\frac\alpha N+2s$ and
\begin{equation} \label{ss2}
\mathcal{G}(u_0)>1,
\end{equation}
 then, $u$ blows-up in finite time.
\end{enumerate}
\end{thm}
\begin{rem}
\begin{enumerate}
\item[1.]
The scattering of global solutions to \eqref{S} is considered in a paper in progress;
\item[2.]
the condition $p<1+\frac\alpha N+2s$ is due to the localized variance identity.
\end{enumerate}
\end{rem}
The next sub-section contains some standard tools.
\subsection{Useful estimates}
First, recall a Hardy-Littlewood-Sobolev inequality \cite{el}.
\begin{lem}\label{hls}
Let $N\geq1$, $0 <\lambda < N$ and $1<s,r<\infty$ be such that $\frac1r +\frac1s +\frac\lambda N = 2$. Then,
$$\int_{\R^N\times\R^N} \frac{f(x)g(y)}{|x-y|^\lambda}\,dx\,dy\leq C(N,s,\lambda)\|f\|_{r}\|g\|_{s},\quad\forall f\in L^r,\,\forall g\in L^s.$$
\end{lem}
The next consequence \cite{st}, is adapted to the Choquard problem.
\begin{cor}\label{cor}\label{lhs2}
Let $N\geq1$, $0 <\lambda < N$ and $1<s,r,q<\infty$ be such that $\frac1q+\frac1r+\frac1s=1+\frac\alpha N$. Then,
$$\|(I_\alpha*f)g\|_{r'}\leq C(N,s,\alpha)\|f\|_{s}\|g\|_{q},\quad\forall f\in L^s, \,\forall g\in L^q.$$
\end{cor}
Sobolev injections \cite{co} give a meaning to several computations done in this note.
\begin{lem}\label{sblv}
Let $N\geq2$, $p\in(1,\infty)$ and $s\in(0,1)$, then 
\begin{enumerate}
\item[1.]
$H^s\hookrightarrow L^q$ for any $q\in[2,\frac{2N}{N-2s}]$;
\item[2.]
the following injection $H^s_{rd} \hookrightarrow\hookrightarrow L^q$ is compact for any $q\in(2,\frac{2N}{N-2s})$;
\item[3.]
for all $\frac12 < \mu  <\frac N2$,
\begin{equation}\label{frcs}
\sup_{x\neq0}|x|^{\frac N2-\mu}|u(x)|\leq C(N,\mu)\|(-\Delta)^\frac\mu2 u\|,\quad\forall u\in H^\mu_{rd}(\R^N ).\end{equation}
\end{enumerate}
\end{lem}
The following fractional Gagliardo-Nirenberg inequality will be useful \cite{mp}.
\begin{lem}\label{fgg}
Let $0<2s<N$. Then, there exists a positive constant $C:=C(N,s)>0$ such that, 
$$\|\cdot\|_p\leq C\|\cdot\|^{1-\frac Ns(\frac12-\frac1q)}\|(-\Delta)^{\frac s2}\cdot\|^{\frac Ns(\frac12-\frac1q)}.$$
\end{lem}
One will use the next fractional chain rule \cite{cw}.
\begin{lem}\label{chain}
Let $s\in(0,1]$ and $1<p,p_i,q_i<\infty$ satisfying $\frac1p=\frac1{p_i}+\frac1{q_i}$. Then,
\begin{enumerate}
\item[1.]
if $G\in C^1(\C)$, then
$$ \|(-\Delta)^{\frac s2}[G(u)]\|_{p}\lesssim \|G'(u)\|_{p_1}\|(-\Delta)^{\frac s2}u\|_{q_1};$$
\item[2.]
$$\|(-\Delta)^{\frac s2}(uv)\|_{p}\lesssim \|(-\Delta)^{\frac s2}u\|_{p_1}\|v\|_{q_1}+\|(-\Delta)^{\frac s2} v\|_{p_2}\|u\|_{q_2}.$$
\end{enumerate}
\end{lem}
\begin{defi}
A couple of real numbers $(q,r)$ is said to be admissible if 
$$q\geq2,\quad r\in[2,\infty),\quad (q,r)\neq(2,\frac{4N-2}{4N-3})\quad\mbox{and}\quad N(\frac12-\frac1r)=\frac{2s}q.$$
Denote the set of admissible pairs by $\Gamma$ and $(q,r)\in\Gamma'$ if $(q',r')\in\Gamma$.
\end{defi}
The so-called {radial} Strichartz estimate \cite{gw} ends this section.
\begin{prop}\label{prop2}
Let $N \geq 2$, $\frac{N}{2N-1}\leq s<1$ and $u_0\in L^2_{rd}$. Then
$$\|u\|_{L^q_t(L^r)}\lesssim\|u_0\|+\|i\dot u-(-\Delta)^s u\|_{L^{\tilde q'}_t(L^{\tilde r'})},$$
if $(q, r)$ and $(\tilde q,\tilde r)$ are admissible pairs.
\end{prop}
\section{Existence of ground states}
This section is devoted to prove Theorem \ref{t3} about the existence of ground states. Let us give some intermediary results.
\begin{lem}
For any $u\in H^s$ and $\lambda>0$
$$\min\{H(u),\partial_\lambda H(u^\lambda)\}>0.$$
\end{lem}
\begin{proof}
A direct computation gives 
\begin{gather*}
H(u)=\|u\|^2+\frac{B-2}{2p}\int_{\R^N}(I_\alpha*|\cdot|^b|u|^p)|x|^b|u|^p\,dx;\\
H(u^\lambda)=\|u\|^2+\frac{B-2}{2p}\lambda^{\frac{2sB}N}\int_{\R^N}(I_\alpha*|\cdot|^b|u|^p)|x|^b|u|^p\,dx;\\
\partial_\lambda H(u^\lambda)=\frac{B(B-2)}{Np}\lambda^{\frac{2sB}N-1}\int_{\R^N}(I_\alpha*|\cdot|^b|u|^p)|x|^b|u|^p\,dx.
\end{gather*}
\end{proof}
The constraint is positive if its quadratic part vanishes.
\begin{lem} \label{K>0}
 Let $0 \neq u_n$ be a bounded sequence of $H^s$ such that
$$ \lim_n\|(-\Delta)^{\frac s2}u_n\| =0.$$
Then, there exists $n_0\in \N$ such that $K(u_n) >0$ for all $n\geq n_0.$
\end{lem}
\begin{proof}
Since $B>2$,
\begin{eqnarray*}
\int_{\R^N}(I_\alpha *|\cdot|^b|u_n|^p)|x|^b|u_n|^p\,dx
&\leq&C\|u_n\|^A\|(-\Delta)^{\frac s2} u_n\|^B\\
&=&o\Big(\|(-\Delta)^{\frac s2}u_n\|^2\Big).
\end{eqnarray*}
Thus, when $n\to\infty$,
$$K(u_n)\simeq\frac{4s}N\|(-\Delta)^{\frac s2}u_n\|^2> 0 . $$
\end{proof}
The minimizing problem \eqref{min} can be expressed with a negative constraint.
\begin{lem}\label{Lemma} 
One has
$$m = \inf_{0\neq u\in H^s}\{H(u)\quad\mbox{s.\, t.}\quad K(u)\leq 0\}.$$
\end{lem}
\begin{proof} 
Denoting by $r$ the right hand side of the previous equality, it is sufficient to prove that $m_{a,c}\leq r.$ Take $u\in H^s$ such that $K(u)<0.$ Because $\lim_{\lambda\rightarrow0}\|(-\Delta)^{\frac s2}u^\lambda\|=0,$ by the previous Lemma, there exists $\lambda\in(0,1)$ such that $K({ u}^\lambda)>0.$ With a continuity argument there exists $\lambda_0\in(0,1)$ such that $K ({ u}^{\lambda_0}) = 0,$ then since $\lambda\mapsto H({ u}^\lambda)$ is increasing, one gets
$$m\leq H({ u}^{\lambda_0}) \leq H(u).$$
This closes the proof.
\end{proof}{}
\begin{proof}[Proof of theorem \ref{t3}]
Let $(\phi_n)$ be a minimizing sequence, namely
\begin{equation} \label{suite}
0\neq \phi_n \in H^s,\quad K(\phi_n) = 0\quad \mbox{and}\quad \lim_n H(\phi_n) = \lim_n S(\phi_n) = m.
\end{equation}
With a rearrangement argument via Lemma \ref{Lemma}, we can assume that $(\phi_n)$ is radial decreasing.\\
$\bullet$ $(\phi_n)$ is bounded in $H^s.$\\
Since 
$$H(\phi_n)=\|\phi_n\|^2+\frac{B-2}{2p}\int_{\R^N}(I_\alpha*|\cdot|^b|\phi_n|^p)|x|^b|\phi_n|^p\,dx\to m,$$
it follows that
$$\sup_n\int_{\R^N}(I_\alpha*|\cdot|^b|\phi_n|^p)|x|^b|\phi_n|^p\,dx<\infty.$$
Then, because
$$S(\phi_n)=\|\phi_n\|_{H^s}^2-\frac1p\int_{\R^N}(I_\alpha*|\cdot|^b|\phi_n|^p)|x|^b|\phi_n|^p\,dx\to m,$$
one gets
$$\sup_n\|\phi_n\|_{H^s}<\infty.$$
$\bullet$ The limit of $(\phi_n)$ is nonzero and $m>0.$\\
Taking account of the compact injection in Lemma \ref{sblv}, take
$$ \phi_n \rightharpoonup \phi\quad \mbox{in}\quad H^s$$
and for all $2<p<\frac{2N}{N-2s},$
$$ \phi_n \rightarrow \phi \quad\mbox{in}\quad L^{p}.$$
The equality $K(\phi_n) =0$ implies that
$$ \|(-\Delta)^{\frac s2} \phi_n\|^2=\frac{B}{2p}\int_{\R^N}(I_\alpha*|\cdot|^b|\phi_n|^p )|x|^b| \phi_n|^p\,dx.$$
Assume that $\phi =0$. Thanks to Hardy-Littlewood-Paley inequality, one gets
$$\int_{\R^N}(I_\alpha*|\cdot|^b|\phi_n|^p )|x|^b| \phi_n|^p\,dx \lesssim\||x|^b\phi_n^p\|_{\frac{2N}{\alpha+N}}^{2}.$$
Take $\rho:=(\frac{N}{-b})^-$ and $r:=\frac{2Np}{\alpha+N-2|b|-\epsilon}$. Using H\"older inequality, write
$$\||x|^b\phi_n^p\|_{L^{\frac{2N}{\alpha+N}}(|x|<1)}\leq\||x|^b\|_{L^\rho(|x|<1)}\|\phi_n\|_{r}^p\lesssim\|\phi_n\|_{r}^p.$$
Since $\tilde p<p<p^*$, taking $\epsilon\to0$, it follows that $2<r<\frac{2N}{N-2s}$. Then,
$$\||x|^b\phi_n^p\|_{L^{\frac{2N}{\alpha+N}}(|x|<1)}\to0,\quad\mbox{as}\quad n\to\infty.$$
Similarly, one obtains
$$\||x|^b\phi_n^p\|_{L^{\frac{2N}{\alpha+N}}(|x|>1)}\to0,\quad\mbox{as}\quad n\to\infty.$$
Thus,
$$\int_{\R^N}(I_\alpha*|\cdot|^b|\phi_n|^p )|x|^b| \phi_n|^p\,dx \to0,\quad\mbox{as}\quad n\to\infty.$$
Now, by Lemma \ref{K>0} yields $K(\phi_n)>0$ for large $n$. This contradiction implies that 
$$\phi \neq 0.$$
With lower semi continuity of $\|\cdot\|_{H^s}$, one has
\begin{eqnarray*}
0 &=& \liminf_n K(\phi_n)\\
&\geq& \frac{4s}N \liminf_n\|(-\Delta)^{\frac s2} \phi_n\|^2-\frac{2sB}{Np}\int_{\R^N} (I_\alpha*|\cdot|^b|\phi|^p)|x|^b| \phi|^p\,dx\\
&\geq& K(\phi).
\end{eqnarray*}
Similarly, $H(\phi)\leq m.$ Moreover, with Lemma \ref{Lemma}, one can assume that
$$K(\phi) =0\quad\mbox{ and }\quad S(\phi) = H(\phi)\leq m.$$
 So, $\phi$ is a minimizer satisfying \eqref{suite}. Thus,
$$m=H(\phi)>0.$$
$\bullet$ The limit $\phi$ is a solution to \eqref{grnd}.\\
There is a Lagrange multiplier $\eta \in \R$ such that $S'(\phi) = \eta K'(\phi).$ Thus,
\begin{eqnarray*}
0 = K(\phi)
&=&\partial_\lambda(S(\phi^\lambda))_{|\lambda=1}\\
&=&\langle S'(\phi), \partial_\lambda(\phi^\lambda)_{|\lambda=1}\rangle\\
&=&\eta\langle K'(\phi),\partial_\lambda(\phi^\lambda)_{|\lambda=1}\rangle\\
&=&\eta\partial_\lambda(K(\phi^\lambda))_{|\lambda=1}.
\end{eqnarray*}
Compute
\begin{gather*}
K(\phi^\lambda)=\frac{4s}N\Big(\lambda^{\frac{4s}N}\|(-\Delta)^{\frac s2}\phi\|^2-\frac{B\lambda^{\frac{2sB}N}}{2p}\int_{\R^N}(I_\alpha*|\cdot|^b|\phi_n|^p)|x|^b|\phi_n|^p\,dx\Big);\\
\partial_\lambda(K(\phi^\lambda))_{|\lambda=1}=(\frac{4s}N)^2\Big(\|(-\Delta)^{\frac s2}\phi\|^2-\frac{B^2}{4p}\int_{\R^N}(I_\alpha*|\cdot|^b|\phi_n|^p)|x|^b|\phi_n|^p\,dx\Big);\\
\end{gather*}
Since $K(\phi)=0$, one gets
$$\partial_\lambda(K(\phi^\lambda))_{|\lambda=1}=(\frac{4s}N)^2\frac{B(2-B)}{4p}\int_{\R^N}(I_\alpha*|\cdot|^b|\phi_n|^p)|x|^b|\phi_n|^p\,dx.$$
Thus, $\eta =0$ and $S'(\phi) = 0.$ So, $\phi$ is a ground state.
\end{proof}
\section{Gagliardo-Nirenberg inequality}
This section, one establishes a sharp Gagliardo-Nirenberg type inequality related to the Choquard problem \eqref{S}. The proof follows by \cite{cw} ideas. 
Let us prove \eqref{euler}. Thanks to a Schwartz symmetrization argument, take a minimizing sequence 
$$\beta:=\frac1{C(N,p,b,\alpha,s)}=\lim_nJ(v_n),\quad v_n\in H^s_{rd}.$$
Define the scaling $u^{\lambda,\mu}:=\lambda u(\mu .)$, $\lambda,\mu\in\R$. Then,
\begin{gather*}
\|(-\Delta)^{\frac s2} u^{\lambda,\mu}\|^2=\lambda^2\mu^{2s-N}\|(-\Delta)^{\frac s2} u\|^2;\\
\|u^{\lambda,\mu}\|^2=\lambda^2\mu^{-N}\|u\|^2;\\
\int_{\R^N}(I_\alpha*|\cdot|^b|u^{\lambda,\mu}|^p)|x|^b|u^{\lambda,\mu}|^p\,dx=\lambda^{2p}\mu^{-N-\alpha-2b}\int_{\R^N}(I_\alpha*|\cdot|^b|u|^p)|x|^b|u|^p\,dx.
\end{gather*}
It follows that 
$$J(u^{\lambda,\mu})=J(u).$$
Take $\psi_n:=v_n^{\lambda_n,\mu_n}$, where 
$$\mu_n:=\Big(\frac{\|v_n\|}{\|(-\Delta)^{\frac s2}  v_n\|}\Big)^\frac1s\quad\mbox{and}\quad \lambda_n:=\frac{\|v_n\|^{\frac N{2s}-1}}{{\|(-\Delta)^{\frac s2}  v_n\|}^\frac N{2s}}.$$
So,
$$\|\psi_n\|=\|(-\Delta)^{\frac s2} \psi_n\|=1\quad\mbox{and}\quad \beta=\lim_nJ(\psi_n).$$
Thanks to Lemma \ref{hls} and Sobolev embedding, 
\begin{eqnarray*}
(A_n)
&:=&\int_{\R^N}|x|^b|(I_\alpha*|\cdot|^b|\psi_n|^{p})|\psi_n|^{p}-(I_\alpha*|\cdot|^b|\psi|^{p})|\psi|^{p}|\,dx\\
&\leq&\int_{\R^N}|x|^b|(I_\alpha*[|\cdot|^b(|\psi_n|^{p}-|\psi|^{p})])|\psi|^{p}-(I_\alpha*|\cdot|^b|\psi_n|^{p})[|\psi|^{p}-|\psi_n|^{p}]|\,dx\\
&\lesssim&(\||x|^b\psi^p\|_{L^{\frac{2N}{\alpha+N}}}+\||x|^b\psi_n^p\|_{L^{\frac{2N}{\alpha+N}}})\||x|^b(|\psi|^p-|\psi_n|^p)\|_{L^{\frac{2N}{\alpha+N}}}.
\end{eqnarray*}
Take $\rho:=(\frac N{|b|})^-$ and $r:=\frac{2N}{\alpha+N-\frac{2N}\rho}$. Because $\tilde p<p<p^*$, it follows that $2<rp<\frac{2N}{N-2s}$. So, by compact Sobolev injections via H\"older inequality, one gets
\begin{eqnarray*}
\||x|^b(|\psi|^p-|\psi_n|^p)\|_{L^{\frac{2N}{\alpha+N}}(|x|<1)}
&\leq&\||x|^b\|_{L^\rho(|x|<1)}\||\psi|^p-|\psi_n|^p\|_{L^r(|x|<1)}\\
&\lesssim&\||\psi|^p-|\psi_n|^p\|_{L^r(|x|<1)}\\
&\lesssim&\|[|\psi_n|-|\psi|]\sum_{k=1}^{p-1}|\psi_n|^k|\psi|^{p-k-1}\|_{L^r}\\
&\lesssim&\|\psi_n-\psi\|_{pr}\sum_{k=1}^{p-1}\|\psi_n\|_{pr}^k\|\psi\|_{pr}^{p-k-1}\to0.
\end{eqnarray*}
Similarly, one estimates the previous integrals on $\{|x|>1\}$. Then, 
$$\lim_n(A_n)=0.$$
Thus,
$$J(\psi_n)=\frac1{\int_{\R^N}(I_\alpha*|\cdot|^b|\psi_n|^{p})|x|^b|\psi_n|^{p}\,dx}\rightarrow\frac1{\int_{\R^N}(I_\alpha*|\cdot|^b|\psi|^{p})|x|^b|\psi|^{p}\,dx}.$$
Thanks to the lower semi continuity of $\|\cdot\|_{H^s}$, 
$$\|\psi\|\leq1\quad\mbox{and}\quad \|(-\Delta)^{\frac s2}\psi\|\leq1.$$
So, $J(\psi)< \beta$ if $\|\psi\|\|(-\Delta)^{\frac s2} \psi\|<1$, which gives
$$\|\psi\|=1\quad\mbox{and}\quad \|(-\Delta)^{\frac s2}\psi\|=1.$$
Then,
\begin{gather*}
\psi_n\rightarrow\psi\quad\mbox{in}\quad H^s;\\
\beta=J(\psi)=\frac1{\int_{\R^N}(I_\alpha*|\cdot|^b|\psi|^{p})|x|^b|\psi|^{p}\,dx}.
\end{gather*}
Finally, $\psi$ satisfies \eqref{euler} because the minimizer satisfies the Euler equation
$$\partial_\varepsilon J(\psi+\varepsilon\eta)_{|\varepsilon=0}=0,\quad\forall \eta\in C_0^\infty\cap H^s.$$
Eventually, one proves \eqref{part3}. Let $\psi$ satisfying \eqref{euler} and 
$$C(N,p,b,\alpha,s)=\frac1\beta=\int_{\R^N}(I_\alpha*|\cdot|^b|\psi|^{p})|x|^b|\psi|^{p}\,dx.$$ 
The scaled function
$$\psi=\phi^{\lambda,\mu}:=\lambda\phi(\mu.),\quad \mu=\Big(\frac AB\Big)^\frac1{2s}\quad\mbox{and}\quad \lambda=\Big((\frac{A}B)^\frac{\alpha+2b}{2s}\frac A{2p\beta}\Big)^\frac1{2(p-1)},$$
satisfies
$$(-\Delta)^s\phi+\phi-(I_\alpha*|\cdot|^b|\phi|^p)|x|^b|\phi|^{p-2}\phi=0.$$
Thus, the equalities
$$\|\psi\|=1=\lambda\mu^{-\frac N2}\|\phi\|,$$
give
$$\beta=\frac{A}{2p}(\frac AB)^{-\frac B2}\|\phi\|^{2(p-1)}.$$
This finishes the proof.
\section{Existence of solutions}
This section is devoted to prove the existence and uniqueness of energy solutions to the non-linear Schr\"odinger problem \eqref{S}. A standard fixed point argument is used. Take $u,v$ in the space
$$ B_T(R):=\{w\in \cap_{(q,r)\in\Gamma} L_T^q(W^{s,r})\quad\mbox{s.\,t}\quad \sup_{(q,r)\in\Gamma}\|w\|_{L_T^q(W^{s,r})}\leq R\}, $$ 
equipped with the complete distance
$$d(u,v):=\|u\|_{S_T(\R^N)}:=\sup_{(q,r)\in\Gamma}\|u-v\|_{L_T^q(L^r(\R^N))}.$$
Define the function
$$\phi(u):=e^{-i.(-\Delta)^s}u_0-\int_0^.e^{-i(.-\tau)(-\Delta)^s}[(I_\alpha*|\cdot|^b|u|^p)|x|^b|u|^{p-2}u]\,d\tau.$$
Using Strichartz estimate and Corollary \ref{cor}, one has
\begin{eqnarray*}
d(\phi(u),\phi(v))
&\lesssim& \|(I_\alpha*[|\cdot|^b|u|^p-|\cdot|^b|v|^p)]|x|^b|u|^{p-2}u\|_{L^{q'}_T(L^{r'}(|x|<1))}\\
&+&\|(I_\alpha*|\cdot|^b|v|^p)(|x|^b|u|^{p-2}u-|x|^b|v|^{p-2}v)\|_{L^{q'}_T(L^{r'}(|x|<1))}\\
&+& \|(I_\alpha*[|\cdot|^b|u|^p-|\cdot|^b|v|^p)]|x|^b|u|^{p-2}u\|_{L^{q'_1}_T(L^{r_1'}(|x|>1))}\\
&+&\|(I_\alpha*|\cdot|^b|v|^p)(|x|^b|u|^{p-2}u-|x|^b|v|^{p-2}v)\|_{L^{q'_1}_T(L^{r'_1}(|x|>1))}\\
&\lesssim&(I)+(II)+(III)+(IV),
\end{eqnarray*}
where $(q,r),(q_1,r_1)\in\Gamma$. Take $\mu:=(\frac N{-b})^-$ and $r:=\frac{2Np}{\alpha+N-\frac{2N}\mu}$. Then, $1+\frac\alpha N=\frac2\mu+\frac{2p}{r}$ and using H\"older and Hardy-Littlewood-Paley inequalities, one gets
\begin{eqnarray*}
(II)
&=&\|(I_\alpha*|\cdot|^b|v|^p)(|x|^b|u|^{p-2}u-|x|^b|v|^{p-2}v)\|_{L^{q'}_T(L^{r'}(|x|<1))}\\\\
&\lesssim&\|(I_\alpha*|\cdot|^b|v|^p)|x|^b(|u|^{p-2}+|v|^{p-2})(u-v)\|_{L^{q'}_T(L^{r'}(|x|<1))}\\
&\lesssim&\||x|^b\|_{L^\mu(|x|<1)}^2\|(\|u\|_{r}^{2(p-1)}+\|v\|_{r}^{2(p-1)})\|u-v\|_{r}\|_{L^{q'}(0,T)}\\
&\lesssim&\|(\|u\|_{r}^{2(p-1)}+\|v\|_{r}^{2(p-1)})\|u-v\|_{r}\|_{L^{q'}(0,T)}.
\end{eqnarray*}
Because $p<p_*$, there exists $\delta>0$ such that  $\frac1{q'}=\frac{2p-1}q+\frac1\delta$. Then, taking account of Sobolev embeddings and H\"older inequality, one obtains
\begin{eqnarray*}
(II)
&\lesssim&T^{\frac1\delta}\Big(\|u\|_{L_T^q(L^r)}^{2(p-1)}+\|v\|_{L_T^q(L^r)}^{2(p-1)}\Big)d(u,v)\\
&\lesssim&T^{\frac1\delta}R^{2(p-1)}d(u,v).
\end{eqnarray*}
Similarly, one estimates $(I)$. Taking in the previous computation, $\mu:=(\frac N{-b})^+$, one controls integrals on $\{|x|>1\}$ and gets 
$$(III)+(IV)\lesssim T^{\frac1\delta}R^{2(p-1)}d(u,v).$$ 
Thus,
$$d(\phi(u),\phi(v))\lesssim T^{\frac1\delta}R^{2(p-1)}d(u,v).$$
Moreover, taking $v=0$ in the previous estimate, one gets
$$\|\phi(u)\|_{S_T(\R^N)}\leq C\|u_0\|+C T^{\frac1\delta} R^{2p-1}.$$
It remains to estimate $\|(-\Delta)^{\frac s2}[\phi(u)]\|_{S_T(\R^N)}$. Taking account of Fourrier transform, one can check that
$$(-\Delta)^{\frac s2}|x|^b=C_{N,b}|x|^{b-s}.$$
Using the chain rules in Lemma \ref{chain} via Corollary \ref{cor}, 
\begin{eqnarray*}
(A)
&:=&\|(-\Delta)^{\frac s2}[\phi(u)-e^{-it(-\Delta)^s}u_0]\|_{L_T^{q'}(L^{r'})}\\
&\lesssim&\|(I_\alpha *(-\Delta)^\frac s2(|\cdot|^b|u|^p))|x|^b|u|^{p-2}u+(I_\alpha *|\cdot|^b|u|^p)(-\Delta)^\frac s2(|x|^b|u|^{p-2}u)\|_{L^{q'}_T(L^{r'})}\\
&\lesssim&\|(I_\alpha *[|\cdot|^b(-\Delta)^\frac s2(|u|^p)])|x|^b|u|^{p-2}u+(I_\alpha *|\cdot|^b|u|^p)|x|^b(-\Delta)^\frac s2(|u|^{p-2}u)\|_{L^{q'}_T(L^{r'})}\\
&+&\|(I_\alpha *(|\cdot|^{b-s}|u|^p))|x|^b|u|^{p-2}u+(I_\alpha *|\cdot|^b|u|^p)(|x|^{b-s}|u|^{p-2}u)\|_{L^{q'}_T(L^{r'})}.
\end{eqnarray*}
Thanks to the Chain rule in Lemma \ref{chain} and arguing as previously, one gets
\begin{eqnarray*}
(I_1)+(II_1)
&:=&\|(I_\alpha *[|\cdot|^b(-\Delta)^\frac s2(|u|^p)])|x|^b|u|^{p-2}u\|_{L^{q'}_T(L^{r'}(|x|<1))}\\
&+&\|(I_\alpha *|\cdot|^b|u|^p)|x|^b(-\Delta)^\frac s2(|u|^{p-2}u)\|_{L^{q'}_T(L^{r'}(|x|<1))}\\
&\lesssim& \|\||x|^b\|_{L^\mu(|x|<1)}^2\Big[\|(-\Delta)^\frac s2(|u|^p)\|_{\frac rp}\|u\|_r^{p-1}+\|u\|_r^p\|(-\Delta)^\frac s2(|u|^{p-2}u)\|_{\frac r{p-1}}\Big]\|_{L^{q'}_T}\\
&\lesssim& \|\|u\|_r^{2p-2}\|u\|_{\dot W^{s,r}}\|_{L_T^{q'}}\\
&\lesssim& \|u\|_{L_T^\infty(L^r)}^{2p-2}\|u\|_{L_T^{q'}(\dot W^{s,r})}\\
&\lesssim& T^{1-\frac2q}\|u\|_{L_T^\infty(H^s)}^{2p-2}\|u\|_{L_T^{q}(\dot W^{s,r})}\\
&\lesssim&T^{1-\frac2q}R^{2p-1}.
\end{eqnarray*}
Take the choice $\rho:=(\frac N{s-b})^-$, $r_1:=\frac{2Np}{N+\alpha+2b+2s(p-1)-\epsilon}$ and $a:=(\frac N{-b})^-$, one gets
$$1+\frac\alpha N=\frac{1}{r_1}+\frac1{\rho}+\frac1a+(2p-1)(\frac{1}{r_1}-\frac sN).$$
Since $p<p^*$, it follows that $(2p-1)q_1'<q_1$ and there exists a positive number denoted also $\delta>0$ such that $\frac1{q_1'}=\frac{2p-1}{q_1}+\frac1\delta$. Then, taking account of Hardy-Littlewood-Paley and H\"older inequalities and to the Chain rule in Lemma \ref{chain}, one gets
\begin{eqnarray*}
(III_1)+(IV_1)
&:=&\|(I_\alpha *(|\cdot|^{b-s}|u|^p))|x|^b|u|^{p-2}u\|_{L^{q'}_T(L^{r'}(|x|<1))}+\|(I_\alpha *|\cdot|^b|u|^p)(|x|^{b-s}|u|^{p-2}u)\|_{L^{q'}_T(L^{r'}(|x|<1))}\\
&\lesssim&\|\||x|^b\|_{L^a(|x|<1)}\||x|^{b-s}\|_{L^\rho(|x|<1)}\|u\|_{\frac{Nr_1}{N-sr_1}}^{2p-1}\|_{L^{q_1'}(0,T)}\\
&\lesssim&\|u\|_{L^{(2p-1)q'_1}_T(L^\frac{Nr_1}{N-sr_1})}^{2p-1}\\
&\lesssim&T^{\frac1\delta}\|u\|_{L_T^{q_1}(L^\frac{Nr_1}{N-sr_1})}^{2p-1}.
\end{eqnarray*}
The conditions $N+\alpha+2b-2s>0$ and $p>\bar p$ imply that $2<r_1<N$. So by Sobolev injections, one gets
$$(III_1)+(IV_1)\lesssim T^{\frac1\delta}\|u\|_{L_T^{q_1}(W^{s,r_1})}^{2p-1}\lesssim T^{\frac1\delta}R^{2p-1}.$$
Similarly, one estimates the integrals on $\{|x|>1\}$. Taking $R>C\|u_0\|_{H^s}$, it follows that $\phi$ is a contraction of $B_T(R)$ for some $T>0$ small enough. The fix point is a solution to \eqref{S}. The uniqueness is a consequence of the previous computations via a translation argument. 
\section{Variance type identity}
This section is devoted to prove Theorem \ref{viril}. By taking the time derivative and using \eqref{S}, one gets
$$\frac d{dt} M_\psi [u(t)] = <u(t),[(-\Delta)^s ,i\Gamma_\psi ]u(t)> + <u(t),[-(I_\alpha*|\cdot|^b|u|^p)|x|^b|u|^{p-2} ,i\Gamma_\psi ]u(t)>,$$
where $[X, Y ]:= XY -Y X$ denotes the commutator of $X$ and $Y$. According to computation done in \cite{bhl}, one have
$$<u(t),[(-\Delta)^s ,i\Gamma_\psi ]u(t)>\leq 4s\|(-\Delta)^{\frac s2}u(t)\|^2+R^{-2s}.$$
Let us treat the nonlinear term
\begin{eqnarray*}
(N)
&:=&<u(t),[-(I_\alpha*|\cdot|^b|u|^p)|x|^b|u|^{p-2} ,i\Gamma_\psi ]u(t)>\\
&=&-<u(t),(I_\alpha*|\cdot|^b|u|^p)|x|^b|u|^{p-2}\nabla\psi_R\nabla u>-<u(t),(I_\alpha*|\cdot|^b|u|^p)|x|^b|u|^{p-2}\nabla.(u\nabla{}{\psi_R})>\\
&+&<u(t),\nabla\psi_R\nabla[(I_\alpha*|\cdot|^b|u|^p)|x|^b|u|^{p-2}u]>+<u(t),\nabla.[\nabla\psi_R(I_\alpha*|\cdot|^b|u|^p)|x|^b|u|^{p-2}u(t){}]>\\
&=&-2<u(t),(I_\alpha*|\cdot|^b|u|^p)|x|^b|u|^{p-2}\nabla\psi_R\nabla u>+2<u(t),\nabla\psi_R\nabla[(I_\alpha*|\cdot|^b|u|^p)|x|^b|u|^{p-2}u]>\\
&=&2\int_{\R^N}|u|^2\nabla\psi_R\nabla[(I_\alpha*|\cdot|^b|u|^p)|x|^b|u|^{p-2}]\,dx.
\end{eqnarray*}
By integration by parts
\begin{eqnarray*}
(N)
&=&2\int_{\R^N}|u|^2\nabla\psi_R\nabla[(I_\alpha*|\cdot|^b|u|^p)|x|^b|u|^{p-2}]\,dx\\
&=&-2\int_{\R^N}\nabla(|u|^2)\nabla\psi_R(I_\alpha*|\cdot|^b|u|^p)|x|^b|u|^{p-2}\,dx-2\int_{\R^N}\Delta\psi_R(I_\alpha*|\cdot|^b|u|^p)|x|^b|u|^{p}\,dx\\
&=&-\frac4p\int_{\R^N}(I_\alpha*|\cdot|^b|u|^p)|x|^b\nabla\psi_R\nabla(|u|^{p})\,dx-2\int_{\R^N}\Delta\psi_R(I_\alpha*|\cdot|^b|u|^p)|x|^b|u|^{p}\,dx\\
&=&\frac4p\int_{\R^N}\nabla(I_\alpha*|\cdot|^b|u|^p)\nabla\psi_R|x|^b|u|^{p}\,dx+2(\frac2p-1)\int_{\R^N}\Delta\psi_R(I_\alpha*|\cdot|^b|u|^p)|x|^b|u|^{p}\,dx\\
&+&\frac{4b}p\int_{\R^N}(I_\alpha*|\cdot|^b|u|^p)x.\nabla\psi_R|x|^{b-2}|u|^{p}\,dx.
\end{eqnarray*}
Thanks to the properties of $\psi$, it follows that
\begin{eqnarray*}
(L)
&:=&\int_{\R^N}(I_\alpha*|\cdot|^b|u|^p)x.\nabla\psi_R|x|^{b-2}|u|^{p}\,dx\\
&=&\int_{|x|<R}(I_\alpha*|\cdot|^b|u|^p)|x|^{b}|u|^{p}\,dx+O\Big(\int_{|x|>R}(I_\alpha*|\cdot|^b|u|^p)|x|^{b}|u|^{p}\,dx\Big)\\
&=&\int_{\R^N}(I_\alpha*|\cdot|^b|u|^p)|x|^{b}|u|^{p}\,dx+O\Big(\int_{|x|>R}(I_\alpha*|\cdot|^b|u|^p)|x|^{b}|u|^{p}\,dx\Big).
\end{eqnarray*}
Using the symmetry of $I_\alpha$, yield
\begin{eqnarray*}
(M)
&:=&\frac4p\int_{\R^N}(\nabla I_\alpha*|\cdot|^b|u|^p)\nabla\psi_R|x|^b|u|^{p}\,dx\\
&=&-2\frac{N-\alpha}p\mathcal K\int_{\R^N}(\nabla\psi_R(x)-\nabla\psi_R(y))\frac{x-y}{|x-y|^{N-\alpha+2}}|y|^b|u(y)|^p|x|^b|u(x)|^{p}\,dy\,dx.
\end{eqnarray*}
Take the sets
\begin{gather*}
\Omega:=\{(x,y)\in\R^N\times\R^N,\quad\mbox{s.\,t}\quad R<|x|<2R\quad\mbox{or}\quad R<|y|<2R\};\\
\Omega':=\{(x,y)\in\R^N\times\R^N,\quad\mbox{s.\,t}\quad |x|>2R,\, |y|<R\quad\mbox{or}\quad |x|<R,\,|y|>2R\}.
\end{gather*}
Then, by the properties of $\psi$, it follows that
\begin{eqnarray*}
-\frac{p}{2(N-\alpha)}(M)
&=&\mathcal K\int_{\{|x|<R,\,|y|<R\}}\frac{|y|^b|u(y)|^p|x|^b|u(x)|^{p}}{|x-y|^{N-\alpha}}\,dy\,dx\\
&+&\mathcal K\int_{\Omega\cup\Omega'}(\nabla\psi_R(x)-\nabla\psi_R(y))\frac{x-y}{|x-y|^{N-\alpha+2}}|y|^b|u(y)|^p|x|^b|u(x)|^{p}\,dy\,dx\\
&=&\mathcal K\int_{\R^N\times\R^N}\frac{|y|^b|u(y)|^p|x|^b|u(x)|^{p}}{|x-y|^{N-\alpha}}\,dy\,dx\\
&+&{O\Big(}\int_{\{|x|>R\}\times\R^N}\frac{|y|^b|u(y)|^p|x|^b|u(x)|^{p}}{|x-y|^{N-\alpha}}\,dy\,dx+\int_{\R^N\times\{|y|>R\}}\frac{|y|^b|u(y)|^p|x|^b|u(x)|^{p}}{|x-y|^{N-\alpha}}\,dy\,dx\Big)\\
&+&O\Big(\int_{\{|x|>R\}\times\{|x-y|>\frac R2\}}(\nabla\psi_R(x)-\nabla\psi_R(y))\frac{x-y}{|x-y|^{N-\alpha+2}}|y|^b|u(y)|^p|x|^b|u(x)|^{p}\,dy\,dx\Big)\\
&+&O\Big(\int_{\{|x|>R\}\times\{|x-y|<\frac R2\}}(\nabla\psi_R(x)-\nabla\psi_R(y))\frac{x-y}{|x-y|^{N-\alpha+2}}|y|^b|u(y)|^p|x|^b|u(x)|^{p}\,dy\,dx\Big)\\
&=&\int_{\R^N}(I_\alpha*|\cdot|^b|u|^p)|x|^b|u(x)|^{p}\,dx+O\Big(\int_{\{|x|>R\}}(I_\alpha*|\cdot|^b|u|^p)|x|^b|u|^{p}\,dx\Big).
\end{eqnarray*}
Regrouping previous computations, and using that $\Delta\psi_R (r)=N$ for $r\leq R$,
\begin{eqnarray*}
(N)
&=&-\frac{2(N-\alpha)}{p}\int_{\R^N}(I_\alpha*|\cdot|^b|u|^p)|x|^b|u(x)|^{p}\,dx+O\Big(\int_{\{|x|>R\}}(I_\alpha*|\cdot|^b|u|^p)|x|^b|u(x)|^{p}\,dx\Big)\\
&+&2(\frac2p-1)\int_{\R^N}\Delta\psi_R(I_\alpha*|\cdot|^b|u|^p)|x|^b|u(x)|^{p}\,dx+\frac{4b}p(L)\\
&=&-2\Big(\frac{N-\alpha}{p}+N(1-\frac2p)\Big)\int_{\R^N}(I_\alpha*|\cdot|^b|u|^p)|x|^b|u(x)|^{p}\,dx+O\Big(\int_{\{|x|>R\}}(I_\alpha*|u|^p)|u|^{p}\,dx\Big)\\
&+&2(\frac2p-1)\int_{|x|>R}(\Delta\psi_R-N)(I_\alpha*|\cdot|^b|u|^p)|x|^b|u(x)|^{p}\,dx+\frac{4b}p\int_{\R^N}(I_\alpha*|\cdot|^b|u|^p)|x|^b|u(x)|^{p}\,dx\\
&=&-2\frac{Np-N-\alpha-2b}{p}\int_{\R^N}(I_\alpha*|\cdot|^b|u|^p)|x|^b|u(x)|^{p}\,dx+O\Big(\int_{\{|x|>R\}}(I_\alpha*|\cdot|^b|u|^p)|x|^b|u(x)|^{p}\,dx\Big)\\
&=&-\frac{2sB}{p}\int_{\R^N}(I_\alpha*|\cdot|^b|u|^p)|x|^b|u(x)|^{p}\,dx+O\Big(\int_{\{|x|>R\}}(I_\alpha*|\cdot|^b|u|^p)|x|^b|u(x)|^{p}\,dx\Big).
\end{eqnarray*}
Taking account of Hardy-Littlewood-Sobolev inequality, 
\begin{eqnarray*}
(I)
&:=&\int_{\{|x|> R\}}(I_\alpha*|\cdot|^b|u|^p)|x|^b|u|^p\, dx\\
&\lesssim&\||\cdot|^bu^p\|_{L^\frac{2N}{\alpha+N}(|x|>R)}^2\\
&\lesssim&\Big(\||\cdot|^bu\|_{L^\infty(|x|> R)}^{p-1-\frac\alpha N}\|u\|^\frac{\alpha+N}N\Big)^2.
\end{eqnarray*}
For $\frac12<\mu:=\frac{1+\varepsilon}2<s<\frac N2$, by \eqref{frcs}, one gets
\begin{eqnarray*}
(I)
&\lesssim&\||\cdot|^bu\|_{L^\infty(|x|> R)}^{2(p-1-\frac\alpha N)}\\
&\lesssim&\Big(R^{-\frac N2+\mu+b}\|(-\Delta)^\frac\mu2 u\|\Big)^{2(p-1-\frac\alpha N)}\\
&\lesssim&\frac1{R^{(N-1-2b-\varepsilon)(p-1-\frac\alpha N)}}\|(-\Delta)^\frac\mu2 u\|^{2(p-1-\frac\alpha N)}\\
&\lesssim&\frac1{R^{(N-1-2b-\varepsilon)(p-1-\frac\alpha N)}}\Big(\|u\|^{1-\frac\mu s}\|(-\Delta)^\frac s2 u\|^\frac\mu s\Big)^{2(p-1-\frac\alpha N)}\\
&\lesssim&\frac1{R^{(N-1-2b-\varepsilon)(p-1-\frac\alpha N)}}\|(-\Delta)^\frac s2 u\|^{\frac{1+\varepsilon}{s}(p-1-\frac\alpha N)},
\end{eqnarray*}
{}{where we used \eqref{frcs} in the second inequality and an interpolation estimate in the fourth one}. In summary, one have has
\begin{eqnarray*}
\frac d{dt} M_{\psi_R} [u]
&=& <u,[(-\Delta)^s ,i\Gamma_\psi ]u> + <u,[-(I_\alpha*|\cdot|^b|u|^p)|x|^b|u|^{p-2} ,i\Gamma_\psi ]u>\\
&\leq& 4s\|(-\Delta)^{\frac s2}u\|^2+CR^{-2s}-\frac{2sB}{p}\int_{\R^N}(I_\alpha*|\cdot|^b|u|^p)|x|^b|u(x)|^{p}\,dx\\
&+&O\Big(\int_{\{|x|>R\}}(I_\alpha*|\cdot|^b|u|^p)|x|^b|u|^{p}\,dx\Big)\\
&\leq&2sBE-2s(B-2)\|(-\Delta)^{\frac s2}u\|^2+\frac C{R^{2s}}\\
&+&\frac C{R^{(N-1-\varepsilon-2b)(p-1-\frac\alpha N)}}\|(-\Delta)^\frac s2 u\|^{\frac{1+\varepsilon}{s}(p-1-\frac\alpha N)}.
\end{eqnarray*}
\section{Global/non-global existence of solutions}
In this section, a sharp criteria of finite time blow-up/global well-posedness is given. Let us start with an auxiliary result.
\begin{lem}\label{stbl}
The following conditions are invariant under the flow of \eqref{S},
\begin{enumerate}
\item[1.]
\eqref{ss} and \eqref{ss1};
\item[2]
\eqref{ss} and \eqref{ss2}.
\end{enumerate}
\end{lem}
\begin{proof}
Using the conservation laws via the sharp Gagliardo-Nirenberg inequality \eqref{ineq}, one have
\begin{eqnarray*}\label{xxx}
 E
&= &\|(-\Delta)^{\frac s2} u\|^2 - \frac1p \int_{\R^N}(I_\alpha*|\cdot|^b|u|^p)|x|^b|u|^p\,dx\\ 
&\geq& \|(-\Delta)^{\frac s2} u\|^2 - \frac{C_{N,p,b,\alpha,s}}{p}\|u_0\|^A\|(-\Delta)^{\frac s2} u\|^B.
\end{eqnarray*}
Define the quantities
$$X(t):=\|(-\Delta)^{\frac s2}u(t)\|^2 \quad\mbox{and}\quad\mathcal D:=\frac{C_{N,p,b,\alpha,s}}{p}{\|u_0\|}^A.$$
Then, 
\begin{equation}\label{eqq}
X - \mathcal DX^\frac B2\leq E, \quad \mbox{on}\quad[0,T^*).\end{equation}
The real function defined on $\R^+$ by $f(x):= x - \mathcal Dx^\frac B2,$ has a local maximum at
$$x_1:=\Big( \frac{2}{\mathcal D B}\Big)^{\frac{2}{B-2}}$$
with a maximum value
$$f(x_1)=\Big( \frac{2}{\mathcal D B}\Big)^{\frac{2}{B-2}}\Big( 1-\frac{2}{B}\Big) .$$
Taking account of Pohozaev identities, 
$$\|(-\Delta)^{\frac s2}\phi\|^2=\frac{B}{A} \|\phi\|^2\quad\mbox{and}\quad \int_{\R^N}(I_\alpha*|\cdot|^b |\phi|^p)|x|^b|\phi|^p\,dx =\frac{2p}B\|(-\Delta)^{\frac s2}\phi\|^2.$$
Then,
$$E(\phi) =\frac{B-2}{B}\|(-\Delta)^{\frac s2}\phi\|^2= \frac{B-2}{A}\|\phi\|^2.$$
Using the previous relation, the condition \eqref{ss} is equivalent to
\begin{equation}\label{x}
E(u_0)< \frac{B-2}{A}M(\phi)^{\frac s{s_c}}{}{M(u_0)}^{\frac{s_c-s}{s_c}}.
\end{equation}
Moreover, by \eqref{part3},
\begin{eqnarray}
f(x_1)
&=&\Big(\frac2{B\mathcal D}\Big)^\frac2{B-2}\Big(1-\frac2B\Big)\nonumber\\
&=&\Big(\frac{2p}{{}{C_{N,p,b,\alpha,s}}{}{M(u_0)}^{\frac A2}B}\Big)^\frac2{B-2}\Big(1-\frac2B\Big)\nonumber\\
&=&\Big((\frac{A}{B})^{1-\frac B2}(M(\phi))^{p-1}({}{M(u_0)})^{-\frac A2}\Big)^\frac2{B-2}\Big(1-\frac2B\Big)\nonumber\\
&=&\frac{B-2}A\Big((M(\phi))^{p-1}({}{M(u_0)})^{-\frac A2}\Big)^\frac2{B-2}\nonumber\\
&=& \frac{B-2}{A}\big(M(\phi)\big)^{\frac s{s_c}}\big({}{M(u_0)}\big)^{\frac{s_c-s}{s_c}}.\label{xx}
\end{eqnarray}
The relations \eqref{x} and \eqref{xx} imply that
$${}{E(u_0)}<f(x_1),$$
by the previous inequality and \eqref{eqq}, one has
\begin{equation}\label{ss3} f(\|(-\Delta)^{\frac s2}u(t)\|^2) \leq {}{E(u_0)}<f(x_1). \end{equation}
Next, taking account of the previous computations, 
$$x_1=\frac BA\big(M(\phi)\big)^{\frac s{s_c}}\big(M(u_0)\big)^{\frac{s_c-s}{s_c}}.$$
\begin{enumerate}
\item[1.]
The condition \eqref{ss1} is equivalent to 
\begin{eqnarray*}
\|(-\Delta)^{\frac s2}u_0\|^2
&<&\|(-\Delta)^{\frac s2}\phi\|^2\Big(\frac{M(\phi)}{{}{M(u_0)}}\Big)^\frac{s-s_c}{s_c}\\
&<&\frac BAM(\phi)\Big(\frac{M(\phi)}{{}{M(u_0)}}\Big)^\frac{s-s_c}{s_c}\\
&<&x_1.
\end{eqnarray*}
Then by \eqref{ss3} and the continuity of $t\to \|(-\Delta)^{\frac s2}u(t)\|$, one gets
$$\|(-\Delta)^{\frac s2}u(t)\|^2<x_1 $$
for all time $t\in [0, T^*)$ which gives \eqref{ss1}. Thus, the conditions \eqref{ss} and \eqref{ss1} are invariant under the flow of \eqref{S}.
\item[2.]
The condition \eqref{ss2} is equivalent to 
\begin{eqnarray*}
\|(-\Delta)^{\frac s2}u_0\|^2
&>&\|(-\Delta)^{\frac s2}\phi\|^2\Big(\frac{M(\phi)}{{}{M(u_0)}}\Big)^\frac{s-s_c}{s_c}\\
&>&\frac BAM(\phi)\Big(\frac{M(\phi)}{{}{M(u_0)}}\Big)^\frac{s-s_c}{s_c}\\
&>&x_1.
\end{eqnarray*}
Then by \eqref{ss3} and the continuity of $t\to \|(-\Delta)^{\frac s2}u(t)\|$, one gets
$$\|(-\Delta)^{\frac s2}u(t)\|^2>x_1\quad\mbox{for all time}\quad t\in [0, T^*) $$
Thus, the conditions \eqref{ss} and \eqref{ss2} are invariant under the flow of \eqref{S}.
\end{enumerate}
\end{proof}
\subsection{Global well-posedness}
The first part of Theorem \ref{Blow-up} is a direct consequence of {Lemma \ref{stbl}. Indeed, in such a case, $\sup_{t\in[0,T^*)}\|u(t)\|_{H^s}<\infty$}.
\subsection{Blow-up}
Let us give an intermediate result which follows as in \cite{st}.
\begin{lem}\label{ode}
Assume that $\frac12<s<1$, $E(u_0)\neq0$ and there exist $t_0>0$ and $\delta>0$ such that
$$M_{\psi_R}[u(t)]\leq-\delta\int_{t_0}^t\|(-\Delta)^{\frac s2}u(\tau)\|^2\,d\tau,\quad\forall t\geq t_0.$$
Then, $T^*<\infty$.
\end{lem}
Let us discuss two cases.
\begin{enumerate}
\item[1.]{Case 1: ${E(u_0)}<0$}.\\
Thanks to the variance identity in Proposition \ref{viril}, for large $R>0$,
\begin{eqnarray*}
\frac{d}{dt}M_{\psi_R}[u(t)]
&\leq&2sBE(u_0)-2s(B-2)\|(-\Delta)^{\frac s2}u(t)\|^2\\
&+&C(\frac1{R^{2s}}+\frac1{R^{(N-1-\varepsilon-2b)(p-1-\frac\alpha N)}})\|(-\Delta)^{\frac s2}u(t)\|^{\frac{1+\varepsilon}{s}(p-1-\frac\alpha N)}\\
&\leq&sBE(u_0)-s(B-2)\|(-\Delta)^{\frac s2}u(t)\|^2,
\end{eqnarray*}
where we discuss $\|(-\Delta)^{\frac s2}u(t)\|\leq1$ or $\|(-\Delta)^{\frac s2}u(t)\|>1$. Integrating in time the previous inequality, $M_{\psi_R}[u(t)]<0$, for large time. Thus, integrating in time, one gets
$$M_{\psi_R}[u(t)]\leq-\delta\int_{t_0}^t\|(-\Delta)^{\frac s2}u(\tau)\|^2\,d\tau,\quad\forall t\geq t_0.$$
The previous Lemma closes the proof.
\item[2.]{Case 2: Assume that \eqref{ss}-\eqref{ss1} are satisfied.}\\
 Take $\eta>0$ such that
$${}{E(u_0)}^{s_c}{}{M(u_0)}^{s-s_c}<[(1-\eta)E(\phi)]^{s_c}M(\phi)^{s-s_c}$$
With a direct computation
$$(1-\eta)(B-2)\|(-\Delta)^{\frac s2}u(t)\|^2>B{}{E(u_0)}.$$
By Proposition \ref{viril}, for $O_R(1)\to0$ uniformly in time, 
\begin{eqnarray*}
\frac{d}{dt}M_{\psi_R}[u(t)]
&\leq&2sBE-2s(B-2)\|(-\Delta)^{\frac s2}u\|^2+\frac C{R^{2s}}\\
&+&\frac C{R^{(N-1-\varepsilon-2b)(p-1-\frac\alpha N)}}\|(-\Delta)^\frac s2 u\|^{\frac{1+\varepsilon}{s}(p-1-\frac\alpha N)}\\
&\leq&-2s\eta(B-2)\|(-\Delta)^{\frac s2}u\|^2\\
&+&\frac1{R^{(N-1-\varepsilon-2b)(p-1-\frac\alpha N)}}\|(-\Delta)^{\frac s2}u\|^{\frac{1+\varepsilon}{s}(p-1-\frac\alpha N)}{}{+O_R(1)}\\
&\leq&[-2s\eta(B-2)+O_R(1)]\|(-\Delta)^{\frac s2}u(t)\|^2+O_R(1)\\
&\leq&-s\eta(B-2)\|(-\Delta)^{\frac s2}u(t)\|^2.
\end{eqnarray*}
The proof follows by Lemma \ref{ode} via te fact that $p<1+\frac\alpha N+2s$.
\end{enumerate}



\begin{thebibliography}{99}


\bibitem{bhl}{\bf T. Boulenger, D. Himmelsbach and E. Lenzmann}, {\em Blow-up for fractional NLS}, J. Funct. Anal. 271 (2016), 2569-2603.
\bibitem{co}{\bf Y. Cho and T. Ozawa}, {\em Sobolev inequalities with symmetry}, Commun. Contemp. Math. 11, no. 3 (2009) 355-365.


\bibitem{cw}{\bf M. Christ and M. Weinstein}, {\em Dispersion of small amplitude solutions of the generalized Korteweg-de Vries equation}, J. Funct. Anal. 100 (1991), 87–109.







\bibitem{gw}{\bf Z. Guo and Y. Wang}, {\em Improved Strichartz estimates for a class of dispersive equations in the radial case and
their applications to non-linear Schr\"odinger and wave equations}, J. Anal. Math. 124, no. 1 (2014), 1-38. 

\bibitem{gz}{\bf Z. Guo and S. Zhu}, {\em Sharp threshold of blow-up and scattering for the fractional Hartree equation}, J. Diff. Eq. 264, no. 4, 15 (2018), 2802-2832.


\bibitem{hr}{\bf J. Holmer and S. Roudenko}, {\em A sharp condition for scattering of the radial 3D cubic non-linear Schr\"odinger equations}, Comm. Math. Phys. 282, (2008), 435-467.
\bibitem{lskn}{\bf  N. Laskin}, {\em Fractional Schr\"odinger equation}, Phys. Rev. E 66 (2002), 056108.


\bibitem{el}{\bf E. Lieb}, {\em Analysis, 2nd ed., Graduate Studies in Mathematics}, Vol. 14, American Mathematical Society, Providence, RI, 2001. 

\bibitem{ll}{\bf  P.-L. Lions}, {\em The Choquard equation and related questions}, Non-linear Anal., 4 (1980), 1063-1072.
\bibitem{mp}{\bf C. Morosi and L. Pizzocchero}, {\em On the constants for some fractional Gagliardo-Nirenberg and Sobolev inequalities}, Expo. Math. 36 (2018), 32-77.



\bibitem{pz}{\bf C. Peng and D. Zhao}, {\em Global existence and blowup on the energy space for the inhomogeneous fractional nonlinear schr\"odinger equation}, Discrete. Cont. Dyn. Syst. B. 24, no. 7 (2019), 3335-3356.

\bibitem{st}{\bf T. Saanouni}, {\em A note on the fractional Schr\"odinger equation of Choquard type}, J. Math. Anal. Appl. 470 (2019), 1004-1029.
\bibitem{st2}{\bf T. Saanouni}, {\em Sharp threshold of global well-posedness vs finite time blow-up for a class of inhomogeneous Choquard equations}, J. Math. Phys. 60, 081514 (2019).




\end{thebibliography}
\end{document}